\documentclass{amsart}
\usepackage{amssymb}
\usepackage{graphicx,enumerate}
\begin{document}
\newtheorem{claim}{Claim}
\newtheorem{proposition}{Proposition}
\newtheorem{definition}{Definition}
\newtheorem{theorem}{Theorem}
\newtheorem{lemma}{Lemma}
\newtheorem{remark}{Remark}
\newtheorem{convention}{Convention}
\def\Rset{\mathbb{R}}
\def\M{\mathcal{M}}
\def\A{\mathcal{A}}
\def\T{\mathcal{T}}
\def\pdw#1{\mathrm{pdw}_{#1}}
\def\Zset{\mathbb{Z}}

\title[A special spherical tiling by congruent concave quadrangles ]{
 Classification of spherical tilings by congruent quadrangles over
 pseudo-double wheels~(I)\\
| a special tiling by congruent concave quadrangles}

\date{\today}                


\author[Y. Akama]{Yohji Akama}
\address{Mathematical Institute\\
  Graduate School of Science \\
  Tohoku University\\
  Sendai 980-0845, Japan} 

\email{akama@m.tohoku.ac.jp}
\urladdr{http://www.math.tohoku.ac.jp/akama/stcq}

\begin{abstract}Every simple quadrangulation of the sphere is
 generated by a graph called a pseudo-double wheel with two local
expansions~(Brinkmann et al. ``Generation of simple quadrangulations of
the sphere.'' Discrete Math., Vol.~305, No.~1-3, pp.~33-54, 2005). So,
toward classification of the spherical tilings by congruent quadrangles,
we propose to classify those with the tiles being convex and the graphs
being pseudo-double wheels. In this paper, we verify that a certain
series of assignments of edge-lengths to pseudo-double wheels does not
admit a tiling by congruent convex quadrangles. Actually, we prove the
series admits only one tiling by twelve congruent \emph{concave}
quadrangles such that the symmetry of the tiling has only three
perpendicular 2-fold rotation axes, and the tiling seems new.
\end{abstract}
\maketitle
\section{Introduction}

A \emph{spherical tiling by congruent polygons} is, by definition, a
covering of the unit sphere by congruent spherical polygons such that
(i) none of the polygons share their inner point, (ii) an edge of a
spherical polygon matches an edge of another spherical polygon, and
(iii) each vertex is incident to more than two edges. Each spherical
polygon is called a \emph{tile}.


If there is a spherical tiling by $p$-gons, then $p$ is either $3,4$ or
$5$~(\cite{agaoka:quad}), because of Euler's formula $V-E+F=2$. In
\cite{MR1954054}, Ueno-Agaoka classified all the spherical tilings by
congruent triangles. By using their classification, we can complete the
classification of all the spherical tilings by congruent quadrangles
where the quadrangle can be divided into two congruent
triangles~\cite{akama12:_spher_tilin_by_congr_rhombi_kites_darts}.
Interestingly, there is a spherical tiling by 24 congruent kites with
the graph being the same as \emph{deltoidal icositetrahedron}, a
\emph{Catalan-solid}~\cite{MR2410150}, while there is another spherical
tiling by the congruent kites with different symmetry.

As for spherical tilings by congruent quadrangles, if the quadrangle
cannot be divided into two congruent triangles, then the edge-lengths of
the tile is necessarily one of two types of Figure~\ref{fig:type24_map},
because an edge of a tile should match an edge of another tile with the
same length and because Euler's formula implies the existence of a
3-valent vertex~\cite{agaoka:quad}.
\begin{figure}[ht]
{\includegraphics[width=6cm]{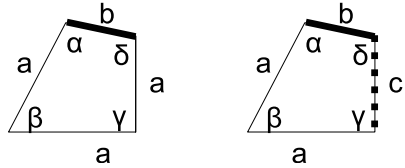}}
\caption{A quadrangle of type~2~(left) and a quadrangle of
type~4~(right). Designation of edge-lengths and angles. The variables
$a,b,c$ are for edge-lengths and they have different values.  The
variables $\alpha,\beta,\gamma,\delta$ are for inner angles.
\label{fig:type24_map}}
\end{figure}
Following \cite{agaoka:quad}, we call the two types \emph{type~2} and
\emph{type~4}. They gave a somehow unexpected, sporadic, 4-fold
rotaional symmetric spherical tiling by 16 congruent quadrangles
of type~2~\cite[Theorem~2]{agaoka:quad}.  As a necessary condition for
a spherical quadrangle of type~2 to exist, they provided inequalities
involving trigonometric functions, in terms of the inner angles of the
tile~(\cite[Proposition~3]{agaoka:quad}). In \cite[Figure~11]{agaoka:quad}, they presented spherical tilings by congruent \emph{concave} quadragles of type~2, and suggested
that the class of  spherical tilings by congruent concave quadrangles is
difficult to classify.

In order to classify all the spherical tilings by congruent
quadrangles, a graph-theoretically systematic approach is first to
classify all the following spherical tilings by $F\ge6$ congruent
quadrangles, having $F/2$-fold rotational symmetry~(See
Figure~\ref{fig:example} and the first three spherical tilings in \cite[Figure~11]{agaoka:quad}).
\begin{figure}[ht]
\begin{tabular}{l l l}
\includegraphics[width=5cm,height=3cm]{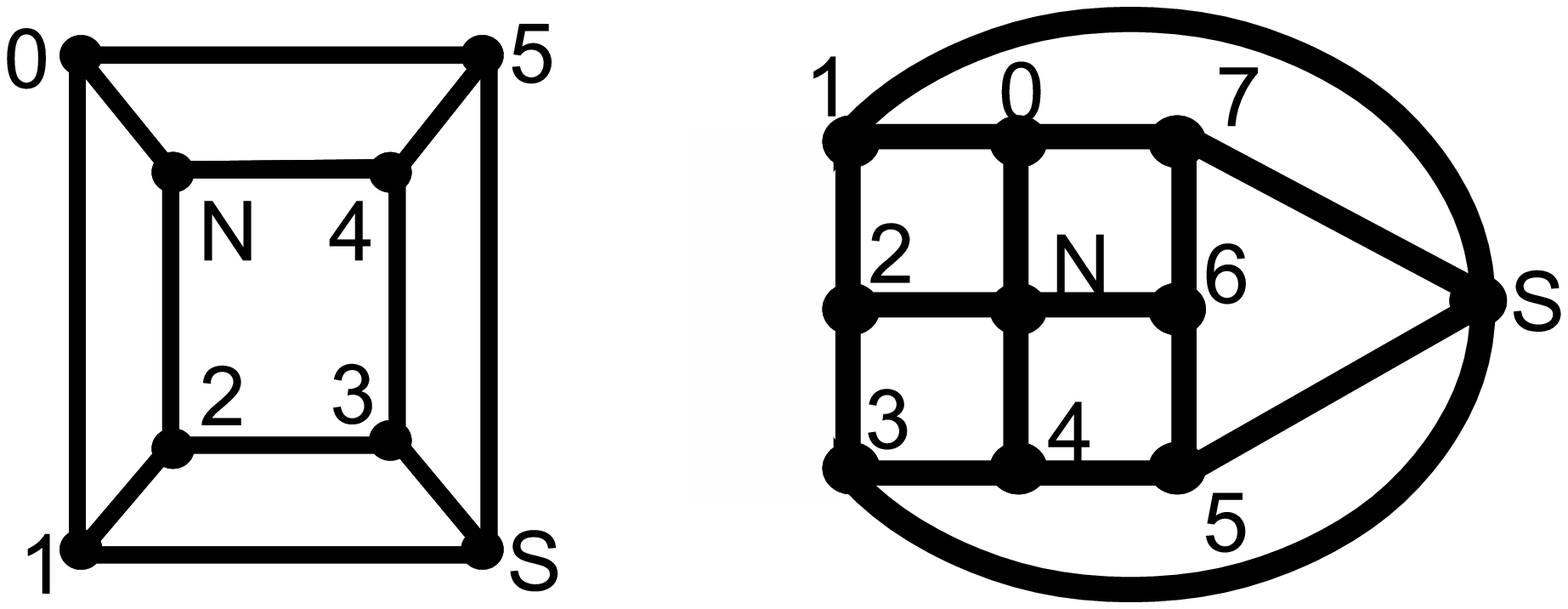}
 &
\includegraphics[width=5cm,height=2.5cm]{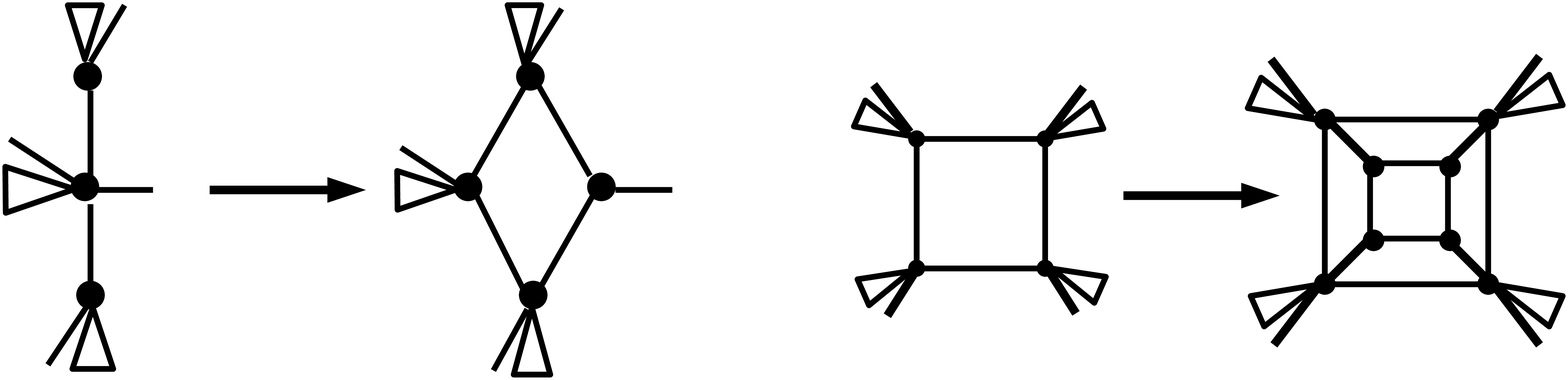}\end{tabular}

\includegraphics[width=10cm]{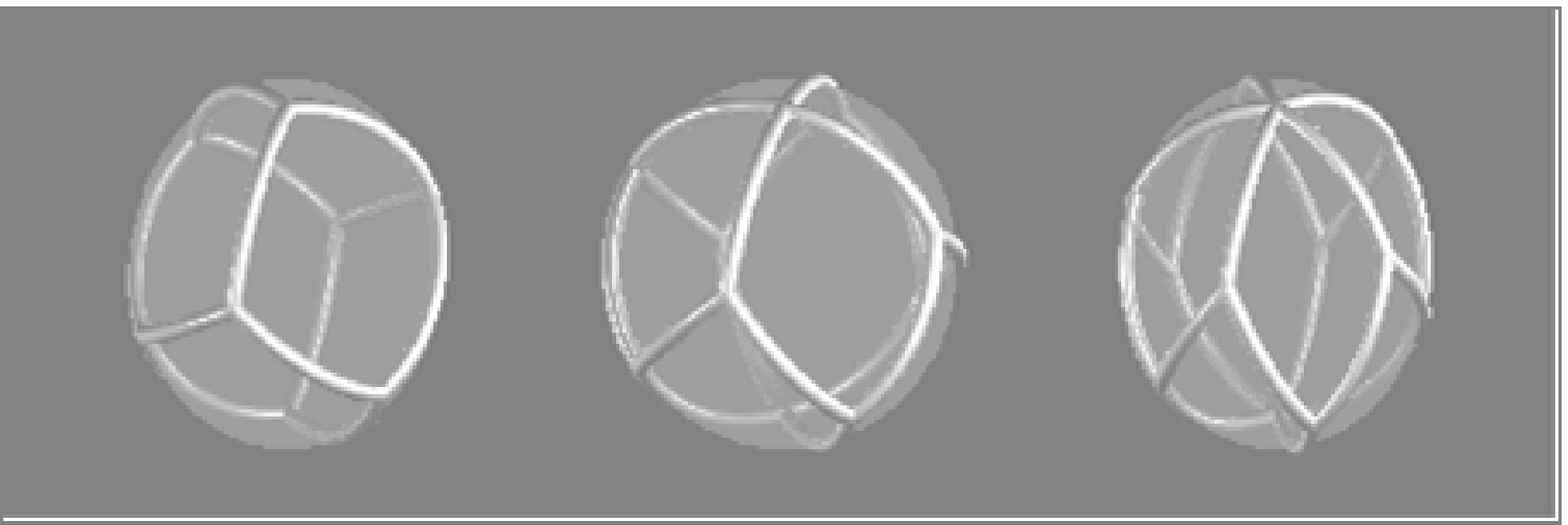}
\caption{The pictures in the
 upper row are,  from left to right, a  pseudo-double wheel of 6 faces, a
 pseudo-double wheel of 8 faces, and two expansions of spherical quadrangulations to increase the number
 of faces (Brinkmann et al.~\cite{MR2186681}). The graphs in this
 paper should be interpreted using Convention~\ref{conv} stated at the
 end of Section~\ref{sec:preliminaries}.  The lower pictures are
spherical tilings by congruent quadrangles over pseudo-double wheel of $F$ faces~($F=6,8,10$, left-to-right)\label{fig:example}. 
 \label{p_expansion}}
\end{figure}
The pair of the vertices and the edges of the tiling forms a graph,
which is obtained from the cycle consisting of $F$ vertices, by
adjoining the north and the south poles alternately to the vertices of
the cycle.  Such a graph falls into the class of \emph{pseudo-double
wheels}.  According to Brinkmann et al.~\cite{MR2186681}, every
spherical quadrangulation is obtained from a pseudo-double wheel of even
or odd faces through finite applications of two local expansions~(see
Figure~\ref{p_expansion}~(upper right)). 

In \cite{akama11:_spher_tilin_by_congr_quadr_pdw_iii}, we prove that for
every spherical tiling by congruent \emph{convex} quadrangles of type~2
or type~4, the edge-lengths of the graph are none of two graphs in
Figure~\ref{fig:forbidden}. As a direct consequence, we can prove the
following statement: If a spherical tiling by congruent convex
quadrangles of type~2 or type~4 has a pseudo-double wheel as a graph,
then the length-assignment of the graph must be $F/2$-fold or an $F/6$-fold ``rotational
symmetric.'' 

The main theorem of this paper is: If a spherical tiling by congruent
quadrangles of type~2 or type~4 over a pseudo-double wheel has an
$F/6$-fold ``rotational symmetric'' length-assignment, then the tiling
necessarily consists of twelve congruent \emph{concave} quadrangles,
such a tiling exists indeed and uniquely up to mirror image, and the
symmetry of the tiling is low compared to the number of tiles. By this theorem, we can
complete the classification of all the spherical tilings by congruent
\emph{convex} quadrangles over pseudo-double
wheels~(\cite{akama11:_spher_tilin_by_congr_quadr_pdw_iii}).

This paper is organized as follows: In the next section, we present
basic relevant notions of spherical tilings by congruent polygons. In
Section~\ref{sec:main}, we present our main theorem, the background and
the organization of the proof.
The rest of the sections are devoted to the proof of the main theorem.

\section{Preliminaries}
\label{sec:preliminaries}

Throughout this paper, by the sphere we mean the sphere with the center
being the origin and the radius being 1, and  say two figures on the sphere are \emph{congruent} if there is an orthogonal
transformation between them. 

By a \emph{spherical triangle\/} (resp. \emph{spherical quadrangle\/}), we
mean a nonempty, simply connected, closed subset $T$ of the sphere with
the area less than $2\pi$ such that the boundary is the union of
three~(resp. four)  distinct geodesic lines but is not the union of any
two~(resp. three) distinct geodesic lines.

To prove that a spherical triangle indeed exists, we use following:
\begin{proposition}[\protect{\cite[p.~62]{MR1254932}}]\label{prop:vinberg}
If $0<A,B,C<\pi,\ A + B + C >\pi,\ \ -A + B + C <\pi,\ \ A - B + C
<\pi,\ \ A + B - C <\pi$, then there exists uniquely up to orthogonal
transformation a spherical triangle such that all the edges are geodesic
lines and the inner angles are $A, B$ and $C$.
\end{proposition} 

We formalize relevant combinatorial notions of spherical tilings by
congruent polygons.  Please refer \cite{MR2744811} for the terminology
of graph theory.

\begin{definition}\label{def:map}
A \emph{map} is a triple $M=(V,\ E,\ \{O_v\;|\; v\in V\})$ such that
$(V,E)$ is a graph (See \cite[Section~1.1]{MR2744811}) and each $O_v$ is
a cyclic order for the edges incident to the vertex $v$. For a vertex
$v$ of $M$, the set $A_v$ of \emph{angles} around $v$ is
\begin{align*}
A_v:=\{(v_1, v, v_2),  (v_2, v, v_3), \ldots, (v_{n-1}, v, v_n), (v_n, v,
      v_1)\}
\end{align*}  
where the list $v v_1, v v_2, \ldots, v v_n$ is the enumeration of $\{ v
u\;|\; v u \in E\}$ without repetition by the cyclic order $O_v$. We
write an inner angle $(u,v,w)$ by $\angle u v w$.  The \emph{mirror} of
a map $M$ is the map $M^R$ with the cyclic orders reversed. Thus $\angle
u v w$ is an inner angle of the original map, if and only if $\angle w v
u$ is an inner angle of the mirror of the map.\end{definition}

\begin{definition}[pseudo-double wheel~\protect{\cite{MR2186681}}]
For an even number $F$ greater than or equal to 6, a \emph{pseudo-double
wheel\/} $\pdw{F}$ with $F$ faces is a map such that

\begin{itemize}
\item the graph is obtained from a cycle
$(v_0, v_1,
v_2, \ldots, v_{F-1})$, by adjoining a new vertex $N$ to each $v_{2i}$
 $(0\le i<F/2)$ and then by adjoining a new vertex $S$ to each $v_{2i+1}$
 $(0\le i<F/2)$. We identify the suffix $i$ of the vertex $v_i$ modulo
 $F$. 
\item
The cyclic order 
at the vertex $N$ is defined as follows: the edge $N v_{2i+2}$ is next
      to the edge $N v_{2i}$. 
The cyclic order at the vertex $v_{2i}$ ($0\le i\le F/2$) is: the edge
 $v_{2i} N$ is next to the edge $v_{2i} v_{2i+1}$, which is next to the edge $v_{2i}
v_{2i-1}$.
The cyclic order 
at the vertex $S$ is: the edge $S v_{2i-1}$ is next to the edge $S v_{2i+1}$. 
The cyclic order at the vertex $v_{2i+1}$ ($0\le i< F/2$) is:
the edge $v_{2i+1} S$ is next to the edge $v_{2i+1} v_{2i}$, which is
      next to the edge $v_{2i+1}v_{2i+2}$.
\end{itemize}
\end{definition} 
We call each edge $N v_{2i}$ \emph{northern}, each edge $S v_{2i+1}$
\emph{southern}, and the other edges \emph{non-meridian}. The number of
edges is $2F$. See Figure~\ref{fig:example}~(lower) for $\pdw{6}$, $\pdw{8}$ and $\pdw{10}$.

We use frequently the following notion:

\begin{definition}[Chart]
A \emph{chart} is,  by definition, a triple $\A=(M, L, K)$  such that
\begin{itemize}
\item $M$ is a map.  Let $V$ be the vertex set of $M$;

\item $L$ is a \emph{length-assignment}, which is a function from the
      edge set $E$ of $M$ to $\{a,b,c\}$ with $a,b,c\in (0,2\pi)$;

\item $K$ is an \emph{angle-assignment}, which is a function from the set
      $\bigcup_{v\in V}A_v$ of the angles to the set of affine combinations of the
      variables $\alpha,\beta,\gamma,\delta$ over $\Rset$ subject to an equation
\begin{align}
 \sum_{\psi\in A_v}K(\psi) = 2\pi\quad \mbox{for each vertex $v\in V$}. \label{sys:2}
\end{align}
Here $A_v$ is the set of angles around the vertex $v$, as in defined in Definition~\ref{def:map}.
For a vertex
      $v\in V$ and an angle $\psi\in A_v$, $K(\psi)$ is called the \emph{type}
      of the angle $\psi$, and $\sum_{\psi\in A_v} K(\psi)$ is called the
      \emph{vertex type of the vertex $v$}; and
\item The pair of $K$ and $L$ satisfies the constraint described by
      Figure~\ref{fig:type24_map}~(left) (``type~2'') for each face or
      the constraint described by
      Figure~\ref{fig:type24_map}~(right) (``type~4'') for each face.
\end{itemize} 

We say a spherical tiling $\T$ by polygons \emph{realizes} a chart $\A$, provided
that there is an embedding $G$ from $\A$ to the sphere such that
\begin{itemize}
\item
 $G$ is
 a bijection from the vertex set of the chart $\A$ to the that of the tiling $\T$, 

\item $G$
 is a bijection from the edge set of the chart $\A$ to that of the
     tiling $\T$, such that (i) if two edges $u v$ and $u' v'$ of $\A$ has intersection $\{w\}$, then 
the intersection of two edges $G(u v)$ and $G(u' v')$ is $\{G(w)\}$; and
     (ii) if $u v$ and $u' v'$ is disjoint, then $G(u v)$ and $G(u' v')$
     do not intersect.

\item
$G$ preserves the cyclic order $O_v$ of each vertex $v$ of the chart $\A$ to the
orientation of the sphere at $G(v)$. In other words, for any vertex
$v$ of the chart, if we rotate a screw at a vertex $G(v)$ according to
     the cyclic order $O_v$, then the
screw goes outbound from the center of the sphere.

\item
For any angle $\angle u v w$ of the chart $\A$, 
 $\angle G(u) G(v) G(w)$ is an angle of the tiling $\T$ and is $K(\angle u v w)$.

\item
For any edge $u v$ of the chart $\A$, $G(u) G(v)$ is an edge of the tiling $\T$ and has
 length $L(u v )$.
\end{itemize} 
The \emph{mirror} $\A^R$ of a chart $\A=(M,L,K)$ is, by definition, a chart
$(M^R, L, K^R)$ where $K(\angle w v u)= K^R(\angle u v w)$.
Thus if a tiling $\T$ realizes $\A$, then the mirror image of
 $\T$ does the mirror $\A^R$.
\end{definition} 

The charts in Figure~\ref{fig:example} and Figure~\ref{fig:forbidden} are subject to
the following convention:
\begin{convention}[\protect{Brinkmann et al.~\cite{MR2186681}}]\label{conv}
Each displayed vertex is distinct from the others;
Edges that are completely drawn must occur in the cyclic order given in the picture;
Half-edges indicate that an edge must occur at this position in the cyclic order
around the vertex;
A triangle indicates that one or more edges may occur at this position in the cyclic
order around the vertex (but they need not);
If neither a half-edge nor a triangle is present in the angle between two edges in the
picture, then these two edges must follow each other directly in the cyclic order of
edges around that vertex.
\end{convention}

\section{Main theorem and the background} \label{sec:main}
In \cite{sakano10:_towar_class_of_spher_tilin}, Sakano classified the
spherical tilings by six or eight congruent quadrangles of type~2 or of
type~4. His argument is generalized for the case the number of the tiles
is more than ten, as follows: Recall quadrangles of type~2 and those of
type~4, by Figure~\ref{fig:type24_map}.

\begin{theorem}[\protect{\cite{akama11:_spher_tilin_by_congr_quadr_pdw_iii}}]\label{thm:forbidden}The
 two charts in Figure~\ref{fig:forbidden} and their mirrors are
impossible as the chart of a spherical tiling by ten or more congruent
\emph{convex} quadrangles of type~$t$ for each $t=2,4$.  Thick edges
are of length $b$ while the other edges are of length $a$ or $c$.  The
charts should be interpreted by Convention~\ref{conv}.

\begin{figure}[ht]
\includegraphics[width=12cm]{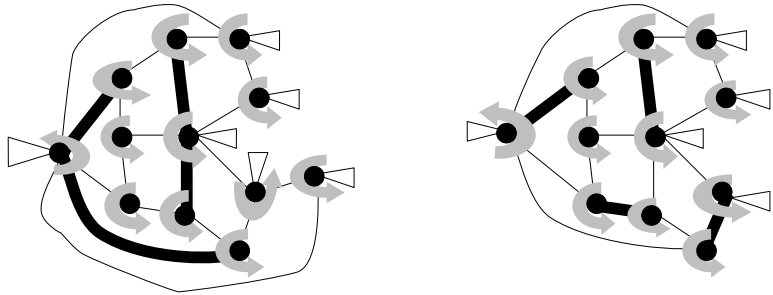} \caption{Charts
forbidden for a spherical tiling by ten or more congruent
\emph{convex} quadrangles of type~$t=2,4$. See
Theorem~\ref{thm:forbidden}. The figures are subject to Convention~\ref{conv}. 
\label{fig:forbidden}}
\end{figure}
\end{theorem}

Let the leftmost vertex of the chart in Figure~\ref{fig:forbidden} be
the vertex $S$ of
the pseudo-double wheel and let the central vertex incident
to at least five edges be the vertex $N$.  In the left chart of
Figure~\ref{fig:forbidden}, an upper, designated,
northern edge of length $b$ is immediately followed by another northern
edge of length $b$.  In the right chart, the designated
northern edge of length $b$ is followed immediately by two consecutive
non-meridian edges of length $b$. By the repeated applications of
Theorem~\ref{thm:forbidden} to a slightly rotated pseudo-double wheel $\pdw{F}$
around the vertices $N$ and $S$, we can observe that all the possibilities of the
length-assignment of
$\pdw{F}$ ($F\ge10$) is as follows:
\begin{enumerate}
\item[(p)] the tile is a convex quadrangle, but no northern edges of $\pdw{F}$ have length $b$; 

\item[(a)] the tile is a convex quadrangle. The northern edges and non-meridian edges
of $\pdw{F}$      alternatingly have length $b$. See
	   Assumption~\eqref{assert:alternating} of Theorem~\ref{thm:k}, and
      Figure~\ref{chart:a}; or

\item[(a$^R$)] (a) with $\pdw{F}$ replaced by the mirror $(\pdw{F})^R$.
\end{enumerate} 
As a consequence of the following Theorem, the possibilities (a) and
(a$^R$) are impossible for any even number $F\ge10$ of tiles. In other
words, for any even number $F\ge10$, there is no spherical tiling by $F$
congruent \emph{convex} quadrangles over $\pdw{F}$ such that (a) or
(a$^R$) holds. The only possibility (p) of the length-assignment of
$\pdw{F}$ is indeed realized with a spherical tiling by $F$ congruent
quadrangles. Such a spherical tiling is obtained from a spherical tiling
by congruent rhombic tiles, such as Figure~\ref{p_expansion}~(lower), by
deforming the tiles to quadrangles of type~2 or of type~4 with the
tiling kept $F/2$-fold rotational symmetric.  For details, see
\cite{akama11:_spher_tilin_by_congr_quadr_pdw_ii,akama11:_spher_tilin_by_congr_quadr_pdw_iii}. So
Theorem~\ref{thm:k}, which is rather a long statement about tiles of
type~2 or type~4, leads to a complete classification of the spherical
tilings by congruent \emph{convex} quadrangles of type~2 or of type~4
over pseudo-double wheels.

\begin{theorem}\label{thm:k}
Assume a chart $\A$ satisfies the following  assumptions:
\begin{enumerate}\renewcommand{\theenumi}{\Roman{enumi}}
\item \label{assumption:2} the map of the chart
     is $\pdw{F}$ for some $F\ge 10$, and

\item \label{assert:alternating} there is $b>0$ such that all the edges $N v_{6i}$, $v_{6i+1} S$
 and $v_{6i+3} v_{6i+4}$ have length $b$ for each nonnegative integer
 $i<F/6$ while the other edges do lengths $\ne b$.
\end{enumerate}
Then there  exists a  spherical tiling $\T$ by congruent
 quadrangles, uniquely up to special orthogonal transformation.
 Moreover the tile is a \emph{concave} quadrangle of type~2, and  the following three conditions hold:
\begin{enumerate}
\item \label{assert:drawn} The tiling $\T$ realizes $\A$, where
the length-assignment and the angle-assignment are
      Figure~\ref{chart:a}~($F=12$ in particular) with the following
      equations:
      \begin{align}
	a&=\arccos\frac{1}{3}, \label{lgh:a}\\
	b&=\arccos\frac{-5}{9},\label{lgh:b}\\
	\alpha&=\arccos\frac{-1}{2\sqrt7},\label{angle:alpha}\\
	\beta&=\frac{\pi}{3},\label{angle:beta}\\
	\gamma&=\frac{4\pi}{3},\label{angle:gamma}\\
	\delta&=\arccos\frac{5}{2\sqrt7}.\label{angle:delta}
      \end{align}

\item \label{thm:l:coordinates} 
Define the spherical polar
coordinate of a point $p$ on the unit sphere to be the pair
$(\theta,\varphi)$ of the length $\theta$
of the geodesic line $N p$ and the \emph{longitude} $\varphi$ of $p$.
The longitude is the angle from the
geodesic line $N v_0$ to $N p$, defined consistently
with the cyclic order of $N$. Let 
\begin{align}
 \phi=\arccos\frac{13}{14}.\label{angle:phi}
\end{align}
Then 
the spherical polar
coordinates of the vertices $v_{i+6}$ is $(\theta,\rho+\pi)$ for $v_i =
      (\theta, \rho )$ ($i=0,1,2,3,4,5$). Moreover  $S=(\pi,0)$, $v_0 = (b, 0)$,  $v_1 = (\pi -
      b, \phi)$, $v_2 = (a, \alpha)$, $v_3 = (\pi - a, \phi + \delta)$,
      $v_4 = (a, \alpha +\beta)$, and $v_5 = (\pi - a, \phi + \delta +\beta)$.

\item \label{D2} The tiling $\T$ has only three perpendicular 2-fold
      rotation axes but no mirror planes. In other words, the
      Sch\"onflies symbol of the tiling is $D_2$. See
 Figure~\ref{fig:perpD2Axis} for the views of $\T$ from the three axes.  
\end{enumerate} 
\renewcommand{\theenumi}{\arabic{enumi}}
\begin{figure}[ht]
\begin{tabular}{l l}
\includegraphics[width=7cm,height=7cm]{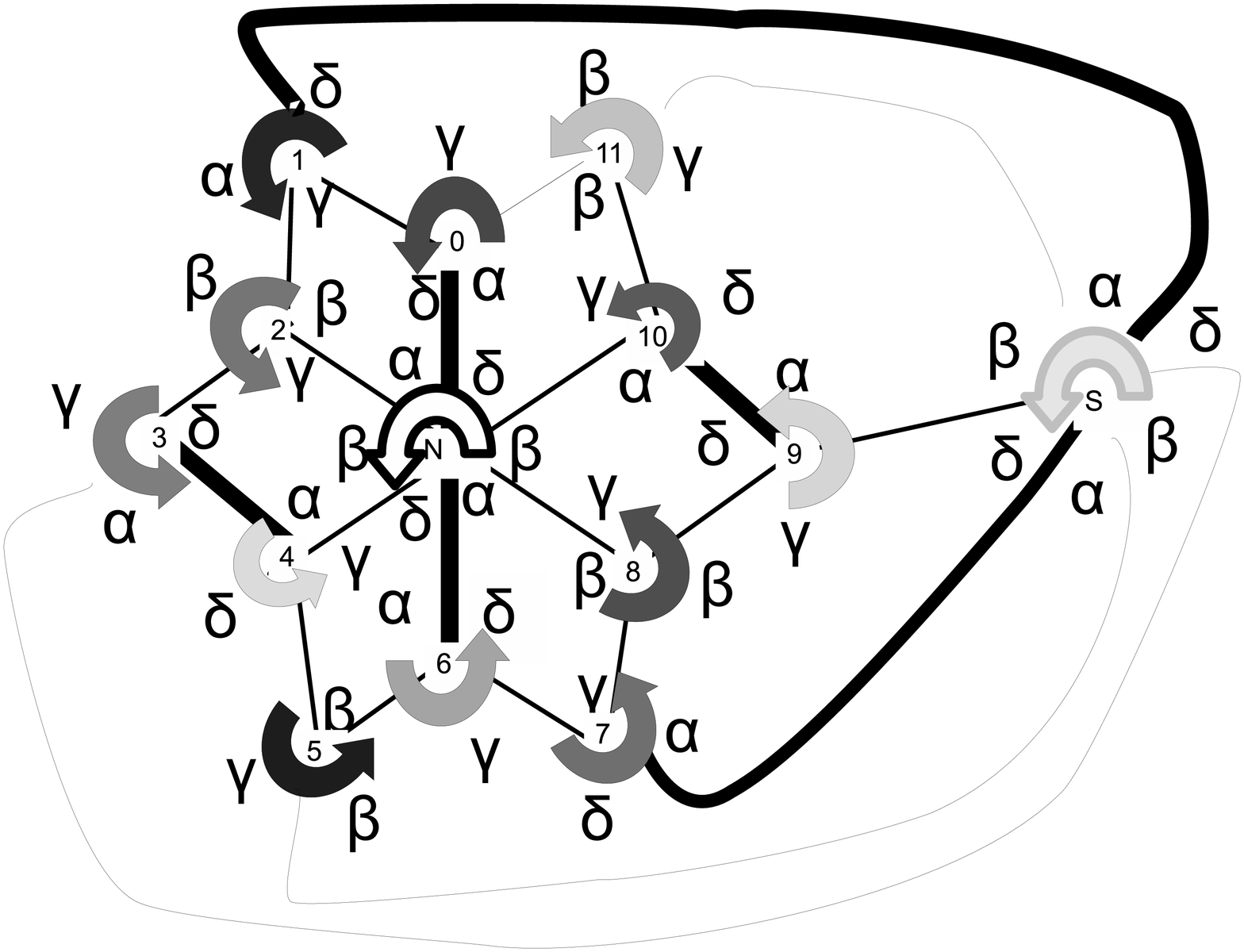}
\end{tabular}
\caption{The chart $\A$ of the tiling.  Thick~(resp. thin) edges of the
 chart correspond to edges of length $b$~(resp. $a$) of the tiling~(see
 Theorem~\ref{thm:k}).  \label{chart:a}}
\end{figure}

\begin{figure}[ht]
\begin{tabular}{c c c}
\includegraphics[width=3cm,scale=0.5]{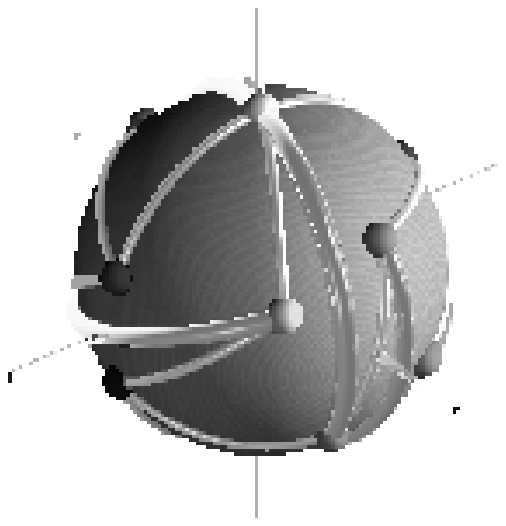}&
\includegraphics[width=3cm,scale=0.5]{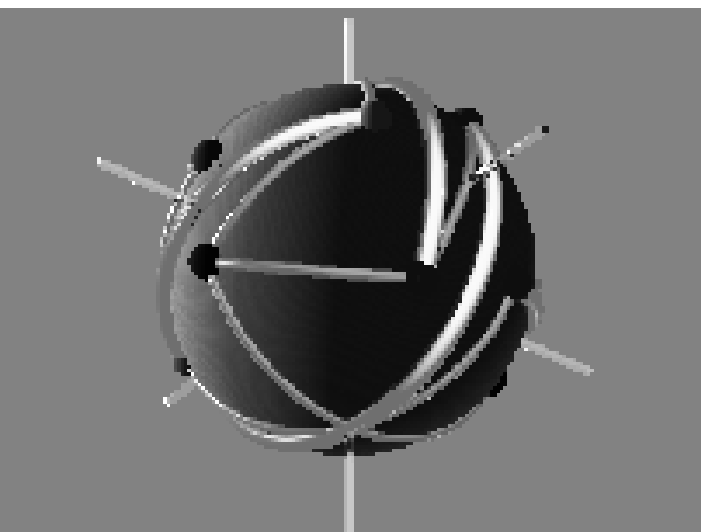}\\
\includegraphics[width=3cm,scale=0.30]{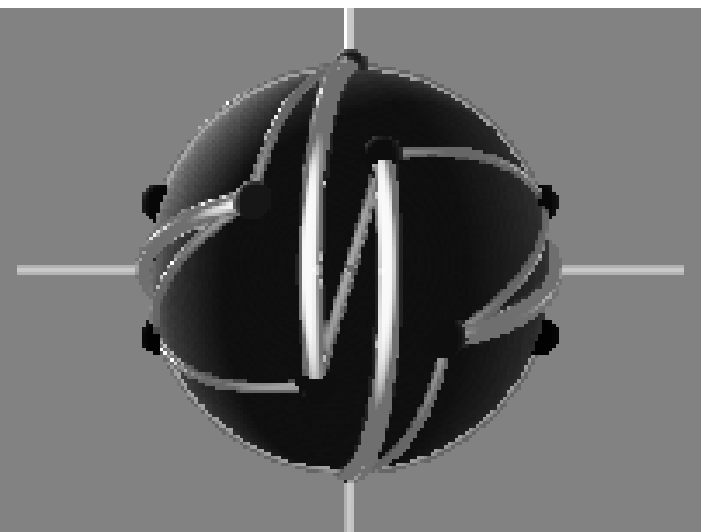}&
\includegraphics[width=3cm,scale=0.30]{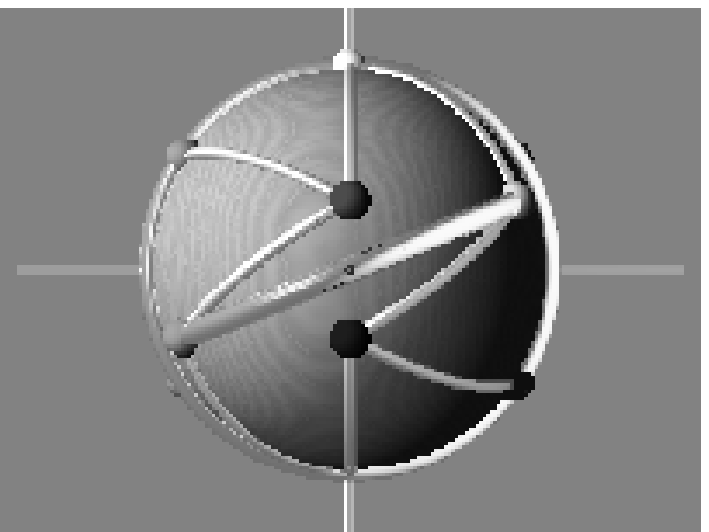}&
\includegraphics[width=3cm,scale=0.30]{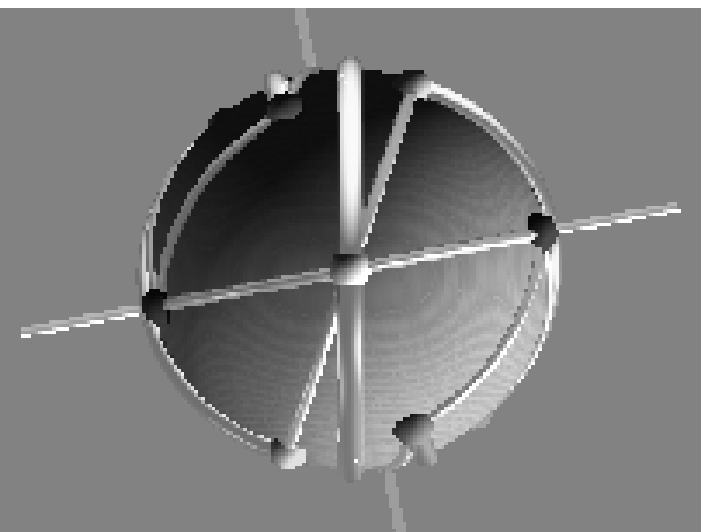}
\end{tabular} 
\caption{The tiling of twelve congruent \emph{concave} quadrangles over
a pseudo-double wheel and the three $2$-fold rotation axes.  The upper
left figure is the view from a general position, and the upper right is
from the antipodal of the first viewpoint. The thick edges are of length
$b$.
The lower left figure is the view from a $2$-fold rotation axis through the midpoint between the vertices $v_0$
 and $v_1$. 
The lower middle figure is from a $2$-fold rotation axis
 through the midpoint between the vertices $v_3$ and $v_4$.
The lower right is from the other $2$-fold rotation axis
 through the poles. In the figure~(lower left), two great circles containing the two middle thick edges
 forms an angle $\phi=\arccos(13/14)$~(see \eqref{angle:phi}). 
 \label{fig:perpD2Axis}}
\end{figure} 
\end{theorem} 

The proof of Theorem~\ref{thm:k} is organized as follows: In
Section~\ref{sec:proof}, by assuming the existence of such a tiling
$\T$, we prove the assertions~\eqref{assert:drawn} and \eqref{D2}
relative to the inner angle $\alpha$ and the number $F$ of tiles, and
prove that the tile is of a concave quadrangle of type~2.
In Section~\ref{subsec:chart}, we determine the absolute values of all
the inner angles, all the edge-lengths of the tile, and $F$.  In
Section~\ref{sec:exist}, by those values, we prove that the tile indeed
exists as a spherical quadrangle, by using
Proposition~\ref{prop:vinberg}.  In Section~\ref{subsec:coordinates}, we
prove the assertion~\eqref{thm:l:coordinates}. In
Section~\ref{subsec:symmetry}, we complete the proof of
Theorem~\ref{thm:k} by establishing the assertion~\eqref{D2}.

\section{Which angles and which edges are equal?}
\label{sec:proof}

We  answer the question in this section by manipulations of systems of equations of
the form \eqref{sys:2} and by elementary geometry.

By an \emph{automorphism} of a map $M$, we mean any
automorphism~\cite[Section~1.1]{MR2744811} $h$ of the graph that
preserves the cyclic orders of the vertices. Here we let $h$ send any
angle $\angle u v w$ to $\angle h(u) h(v) h(w)$.  If $\A=(M,L,K)$ is a
chart and $h$ is an automorphism of the map $M$, then let $h(\A)$ be a
chart $(M,\ L\circ h^{-1},\ K\circ h^{-1})$.

\begin{definition}\label{def:fg}
Let $\A=(M,L,K)$ be a chart satisfying the assumptions of Theorem~\ref{thm:k}, define
 two automorphisms $f$ and $g$ on the map $M$ of $\A$ as follows:
\begin{align}
 f(N)=S,\quad f(S)=N,\quad f(v_i)=v_{1-i}\ (0\le i\le
 F-1); \label{auto:f}\\
 g(N)=N,\quad g(S)=S,\quad g(v_i)=v_{i+6}\ (0\le i\le
 F-1). \label{auto:g}
\end{align}
Here the suffix $i$ of the vertex $v_i$ is understood modulo $F$ as
before.  Then we can prove that both of $f$ and $g$ are indeed automorphisms.\end{definition}

\begin{lemma}\label{lem:combinatorial}
Suppose some tiling $\T$ realizes a chart $\A=(M,L,K)$ satisfying the assumptions of
 Theorem~\ref{thm:k}.
Then,
\begin{enumerate}
\item \label{multiple6} The number $F$ of faces is a multiple of 6
      greater than or equal to 12.

\item \label{type2} The tile is a quadrangle of type~2.

\item \label{invariance}
$\A= f(\A)= g(\A)$.

\item \label{thm:k:case1}
The inner angles satisfy
$\beta=\frac{4\pi}{F},\quad
\gamma=2\pi-\frac{8\pi}{F}>\pi,\quad
\delta= \frac{8\pi}{F} -\alpha$.

\item \label{fundamental}
$K$ has the same action as in
 Figure~\ref{chart:a} on the angles $\angle v_{-2} N v_0$, $\angle v_0
      N v_2$, $\angle v_2 N v_4$ and on all the angles around the vertices
 $v_0, v_2, v_3$. Especially, 
\begin{align}
 \angle v_3 v_2 N = 2\pi - \frac{8\pi}{F}>\pi,\quad \angle v_2 N v_4 =
 \frac{4\pi}{F}, \quad \angle N v_4 v_3 = \alpha. \label{32N2N4}
\end{align}
\end{enumerate}

\end{lemma} 

To prove the assertion~\eqref{multiple6}, assume $F$ is not a multiple
of 6. Then the edge $N v_2$ or $N v_4$ should have length $b$, which
contradicts against Assumption~\eqref{assert:alternating} of
Theorem~\ref{thm:k}. By this and Assumption~\eqref{assumption:2} of
Theorem~\ref{thm:k},  the assertion~\eqref{multiple6} of
Lemma~\ref{lem:combinatorial} follows.

\subsection{Basic general lemmas}
To prove the other assertions of Lemma~\ref{lem:combinatorial}, we associate to the chart in
question a  system of equations and inequalities about the
tile's inner angles $\alpha,\beta,\gamma,\delta$. The system consists of the equations \eqref{sys:2}, inequalities
$0<\alpha,\beta,\gamma,\delta<2\pi$, and an equation
\begin{align}
 \alpha+\beta+\gamma+\delta -2\pi = \frac{4\pi}{F}.\label{eq:area}
\end{align}
The last equation is because all the $F$ tiles are congruent quadrangles
and because the area of a spherical triangle is the sum of the inner
angles subtracted by $\pi$ according to \cite{MR1254932}.

If the system of equations and inequalities associated to a chart $\A$
is unsolvable, no tiling realizes $\A$. To show that such systems are
unsolvable, we  use the following lemmas.

\begin{lemma} \label{lem:6.2}In a quadrangle of type~2, no inner angle
 is equal to the opposite inner angle~(i.e.\@ $\beta\ne\delta$ and
$\alpha\ne\gamma$). In a quadrangle of type~4, we have
$\alpha\ne\gamma$.
\end{lemma}
\begin{proof} Draw a diagonal line in the tiles, and then argue 
with the isosceles triangle(s).  \end{proof}

\begin{lemma}\label{lem:6.3}If a tiling by congruent quadrangle of
 type~4 realizes a chart, then, for each $Z\in
 \{\alpha,\beta,\gamma,\delta\}$, the chart has no 3-valent vertex of type 
 $Z+2\delta$ or $Z+\beta+\delta$.\end{lemma}
\begin{proof}Recall  $\delta$ is the type of an angle between an edge of length $b$
and that of length $c$. Because $b\ne c$, if a 3-valent vertex has vertex
type $2\delta+Z$ for some type $Z$, then $Z$ is the type of an angle
between two edges of length $b$, or is that of an angle between two
edges of length $c$. But such an angle does not exists in any
quadrangles of type~4.

$\beta$ is the type of an angle between two edges of length
 $a$. $\delta$ is the type of an angle between an edge of length $b$ and
 that of length $c$. By
$c\ne a\ne b$, no edge of the angle of type $\beta$ does not match
an edge of the angle of type $\delta$. Thus an angle of type $\delta$
cannot be adjacent to an angle of type $\beta$ in a quadrangle of
type~4. Hence a vertex of type $Z+\beta+\delta$ is impossible.
\end{proof}

\def\vA{\alpha}
\def\vB{\beta}
\def\vC{\gamma}
\def\vD{\delta}

\begin{lemma}\label{lem:isoscelestrapezoid2}
For any quadrangle of type~2,  $\alpha=\delta$ if and only if $\beta=\gamma$.
\end{lemma}
\begin{proof}
Assume $\alpha = \delta$. We identify the vertex of the quadrangle
with the inner angle. Because the quadrangle 
$\alpha\beta\gamma\delta$ is of type~2,  the two edges $\vA \vB$ and $\vC \vD$
have length $a$, and then the triangle $\vA \vB \vD$ is congruent to the triangle
$\vD \vC \vA$. Thus $\angle \vB \vD \vA = \angle \vC \vA
\vD$. As $\alpha = \delta$, we have $\angle \vB \vA \vC = \angle \vC
\vD \vB$. Hence $\angle \vC \vB \vD = \angle \vC \vD \vB = \angle \vB
\vA \vC = \angle \vB \vC \vA$. Thus $\beta = \angle \vA \vB \vD + \angle
\vC \vB \vD = \angle \vD \vC \vA + \angle \vB \vC \vA = \gamma$.
Conversely, assume  the inner angle $\beta$ is equal to the inner
angle $\gamma$. As the triangle $\vA \vB \vC$ is congruent to the
triangle $\vD \vC \vB$, we have $\angle \vB \vA \vC = \angle \vB \vC \vA
= \angle \vC \vB \vD = \angle \vC \vD \vB$. Let the point $P$ be shared
by the two diagonal segments $\alpha\gamma$ and $\beta\delta$ of the
quadrangle. As $\angle P\vB \vC = \angle P\vC \vB$, the length of the
edge $\vB P$ is equal to that of the edge $\vC P$. Since the length of
the diagonal $\vA \vC$ is equal to that of the other diagonal $\vD \vB$,
we have $P \vA = P \vD$, by which the triangle $\vA P \vD$ is an
isosceles triangle. So $\angle \vB \vA \vC = \angle \vC \vD
\vB$. Thus the inner angle $\alpha$ is equal to the inner angle $\delta$.  
\end{proof}

\begin{lemma}\label{lem:aux}
Suppose a tiling by congruent quadrangles of type~2 realizes a chart
such that
(1) there is a vertex incident only to three edges of length $a$, and
(2) there is a 3-valent vertex incident to two edges
      of length $a$ and to one edge of length $b$.
Then  $\alpha\ne\delta$ and $\beta\ne\gamma$. 
\end{lemma}
\begin{proof}Assume otherwise. By Lemma~\ref{lem:isoscelestrapezoid2},
 we have $\alpha=\delta$ and $\beta=\gamma$. So the
assumption~(1) implies $\beta=\gamma=2\pi/3$.  By the
assumption~(2), there exists a vertex of type $X+X'+Y$ $(X,
X'\in \{\alpha,\delta\}, \ Y\in \{\beta,\gamma\})$.  Thus
$\alpha=\beta=\gamma=\delta=2\pi/3$.  This contradicts against
Lemma~\ref{lem:6.2}.\end{proof}

\subsection{The proof of the assertions~\eqref{type2}, \eqref{invariance}, \eqref{thm:k:case1} and \eqref{fundamental} of Lemma~\ref{lem:combinatorial}}
\begin{lemma}\label{lem:u}
If a tiling by congruent quadrangles of type~2 realizes a chart satisfying
the assumptions of Theorem~\ref{thm:k}, then $\alpha\ne\delta$ and
$\beta\ne\gamma$.
\end{lemma}

\begin{proof}
Since the tile $N v_0 v_1 v_2$ of the pseudo-double wheel $\pdw{F}$ in
question is of type~2 and has an edge $N v_0$ of length $b$, other edges
$v_1 v_2$ and $v_2 N$ of the same tile have lengths $a$. The tile $S v_1
v_2 v_3$ is of type~2, and has an edge $S
v_1$ of length $b$ because of Assumption~\eqref{assert:alternating} of
Theorem~\ref{thm:k}. Thus the opposite edge $v_2 v_3$ of the same tile $S v_1
v_2 v_3$ has length $a$. Thus the vertex $v_2$ is incident only to three
edges of length $a$. On the other hand, the vertex $v_1$ is incident to
one edge of length $b$ and to two edges $v_0 v_1$ and $v_0 v_{F-1}$ of
length $a$.  By Lemma~\ref{lem:aux}, we have $\alpha\ne\delta$ and
$\beta\ne\gamma$.  \end{proof}

We can easily  verify the following Lemma for automorphisms $f,g$ defined in Definition~\ref{def:fg}.
\begin{lemma}\label{lem:sym}
If  $\A$ is  any chart satisfying the assumptions of
Theorem~\ref{thm:k}, then
\begin{enumerate}
\item  $f(\A)$ and $g(\A)$  are charts satisfying the assumptions of
       Theorem~\ref{thm:k}.

\item Whenever we can prove a property $P(\A)$,  we can prove the
 properties $P(f(\A))$ and $P(g(\A))$.
\end{enumerate} 
\end{lemma}

\begin{definition}[conjugate]
Fix a chart.
For any property $P(\alpha,\beta,\gamma,\delta)$ for the inner angles $\alpha,
\beta,\gamma,\delta$ of the tile, the \emph{conjugate property}
$P^*(\alpha,\beta,\gamma,\delta)$ is, by definition, a property
$P(\delta,\gamma,\beta,\alpha)$. 
\end{definition}

\begin{lemma}\label{lem:lm}\label{lem:6}
Suppose a tiling by congruent quadrangles realizes a chart satisfying
 the assumptions of Theorem~\ref{thm:k}. Then the tile is of type~2, and
 the chart satisfies the following properties \eqref{prop3star} and
 \eqref{prop4} or satisfies the conjugate properties
 $\eqref{prop3star}^*$ and $\eqref{prop4}^*$.
\begin{enumerate}[(A)]
\item \label{prop3star} For each integer $0\le i<F/6$, the types
of $\angle S v_{6i+1} v_{6i+2}$ and $\angle N v_{6i} v_{6i-1}$ are
$\delta$, those of $\angle v_{6i} v_{6i+1} S$ and $\angle v_{6i+1} v_{6i} N$
are $\alpha$, and those of $\angle v_{6i+2} v_{6i+1} v_{6i}$ and $\angle
      v_{6i-1} v_{6i}
v_{6i+1}$ are $\beta$.

\item \label{prop4} For each  integer $0\le i<F/6$, the types of
      $\angle v_{6i+2} N v_{6i+4}$ and $\angle v_{6i+5} S v_{6i+3}$ are all $\gamma$.
\end{enumerate}
\end{lemma}

\noindent
\begin{proof} It is sufficient to prove the
 following properties and the conjugate properties
Property~$1^*$ and Property~$2^*$:

\medskip 
\emph{Property~1}. It is not the case that: 
for some nonnegative integer $i<F/6$,
both of  $\angle v_{6i+1} v_{6i} N$
      and $\angle N v_{6i} v_{6i-1}$ have type $\delta$ and $\angle v_{6i-1} v_{6i} v_{6i+1}$ has type $\beta$.

\medskip
\emph{Property~2}. It is not the case that: 
for some nonnegative integer $i<F/6$,
both of  $\angle v_{6i+1} v_{6i} N$
      and $\angle N v_{6i} v_{6i-1}$ have type $\alpha$ and $\angle v_{6i-1} v_{6i} v_{6i+1}$ has type $\beta$.

\medskip We verify the sufficiency.  Assume we can prove Property~1,
Property~2, Property~1$^*$ and Property~2$^*$. Then for every
nonnegative integer $i< F/6$, the type of the vertex $v_{6i}$ of
$\pdw{F}$ is $\alpha+\beta+\delta$ or $\alpha+\gamma+\delta$, and thus
the type of the vertex $v_{6i+1}$ of $\pdw{F}$ is so, because the
automorphism $f$ satisfies $f(v_{6i})=v_{6i+1}$ and Lemma~\ref{lem:sym}.

Because every edge $v_{6i+3} v_{6i+4}$ has length $b$, each angle
 mentioned in \eqref{prop4} is $\beta$ or $\gamma$. 

First consider the case all the vertices $v_{6i}$ and $v_{6i+1}$ have type
$\alpha+\gamma+\delta$. Then all the angles $\angle v_{6i+2} N v_{6i+4}$
and $\angle v_{6i+5} S v_{6i+3}$ do $\beta$. Otherwise,
$\alpha+\gamma+\delta<2\pi$ because $\alpha$ and $\delta$ appear as types
of angles around the poles and $F/2\ge 5$. Hence we have $\eqref{prop4}^*$.

Each $\angle v_{6i+1} v_{6i} N$ has type $\delta$. Otherwise
$\angle v_{6i+1} v_{6i} N$ has type $\alpha$ and thus $\angle v_{6i+2}
v_{6i+1} v_{6i}$ does $\beta$, which contradicts against the vertex
type of $v_{6i+1}$ being $\alpha+\gamma+\delta$. Similarly we can prove
$\angle v_{6i} v_{6i+1} S$ has type $\delta$. So we have
$\eqref{prop3star}^*$.

Thus each $\angle v_{6i+4} N v_{6i+6}$ has type $\delta$. If the tile is
of type~4, then we have contradiction against $\eqref{prop4}^*$, by
considering the length of the edge $v_4 N$. So the tile is of type 2.

Next consider the case all the vertices $v_{6i}$ and $v_{6i+1}$ have type
$\alpha+\beta+\delta$. But we can prove \eqref{prop3star} and
\eqref{prop4} by the same argument as above but with
$(\alpha,\beta)\leftrightarrow (\delta,\gamma)$ swapped. 

In the remaining case, we have a vertex of type
$\alpha+\beta+\delta=2\pi$ and that of type
$\alpha+\gamma+\delta=2\pi$. The former vertex leads to a contradiction
against Lemma~\ref{lem:6.3}, when the tile is of type~4. If the tile is
 of type~2, this
case implies $\beta=\gamma$, which contradicts against
Lemma~\ref{lem:u}.

Hence to verify Lemma~\ref{lem:lm},
below we prove Property~1, Property~2 and their conjugate properties.

Property~$1^*$ holds when the tile is of type~4, by the following
argument. If both of $\angle v_{6i+1} v_{6i} N$ and $\angle N v_{6i} v_{6i-1}$ have
type $\alpha$, then the length of the edge $v_{6i} v_{6i+1}$ and that of edge
$v_{6i} v_{6i-1}$ are both $a$. So $\angle v_{6i-1} v_{6i} v_{6i+1}$ has type
$\beta$.

Property~$1$ and Property~$2^*$ hold when the tile is of type~4, 
by Lemma~\ref{lem:6.3}.

Property~$1$ holds when the tile is of type~2.  Otherwise,
$\beta+2\delta=2\pi$. See Figure~\ref{fig:4}~(left). The type of $\angle
v_{6i-1} v_{6i-2} N$ is $\beta$. By
Assumption~\eqref{assert:alternating} of Theorem~\ref{thm:k}, the edge
$v_{6i-2} v_{6i-3}$ has length $b$, the type of the vertex $v_{6i-2}$ is
one of $\alpha+\beta+\delta$, $\beta+2\alpha$, and $\beta+2\delta$. For
the first two cases, we have $\alpha=\delta$ which contradicts against
Lemma~\ref{lem:u}. Therefore in the vertex type of the vertex $N$, the
type $\gamma$ occurs.  However the vertex type of $v_{6i-3}$ is
$2\alpha+Y$ for some $Y\in\{\beta,\gamma\}$. In case $Y=\beta$, we have
$\alpha=\delta$, which  contradicts against Lemma~\ref{lem:u}. In the other
case $Y=\gamma$, the vertex type of $v_{6i-3}$ is a proper subtype of
$N$, which is a contradiction.

We can prove Property~$1^*$ holds when the tile is of type~2, by swapping $(\alpha,\beta)\leftrightarrow (\delta,\gamma)$ in the proof above.

\begin{figure}[ht]
\begin{center}
\includegraphics[width=6cm]{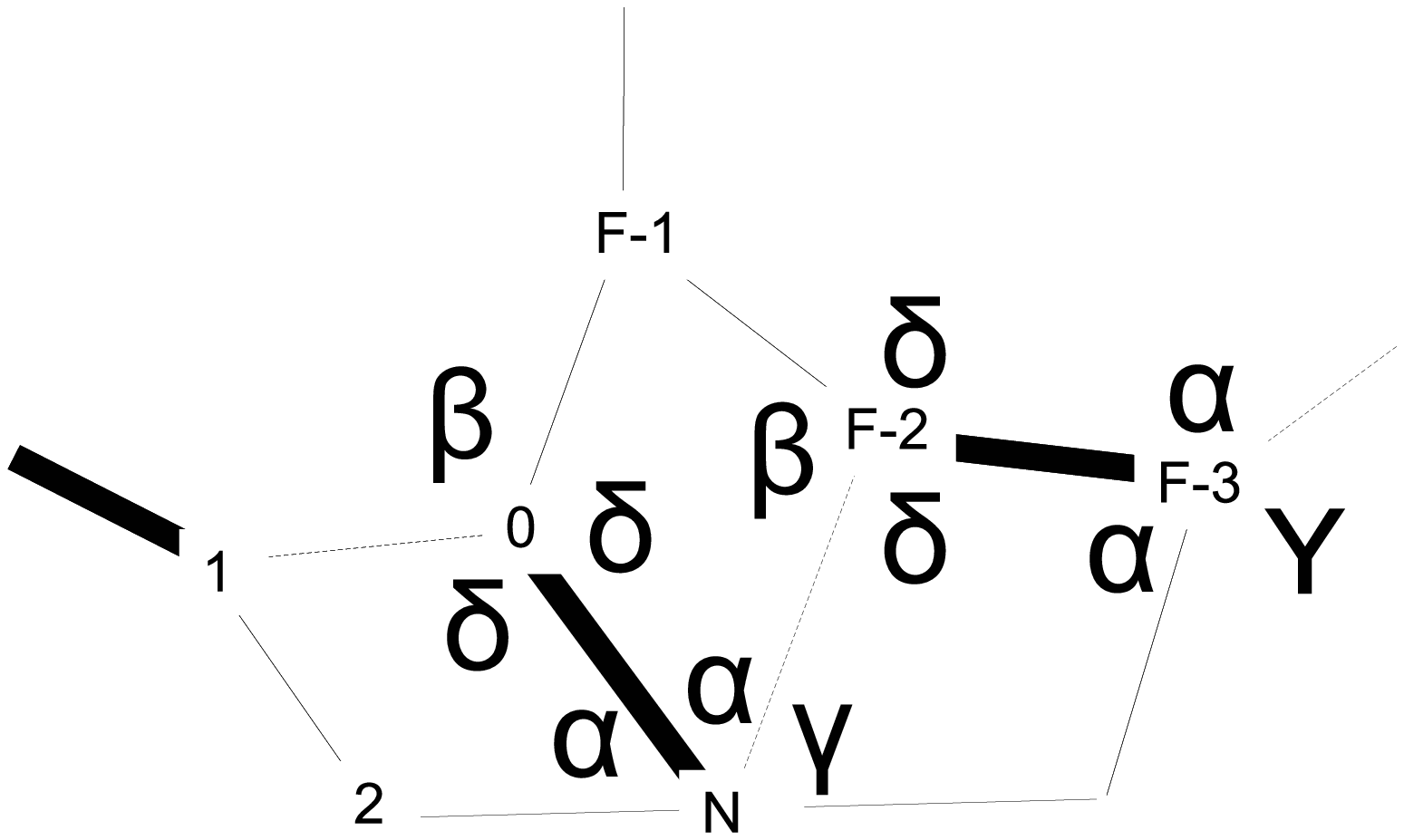}
\includegraphics[width=6cm]{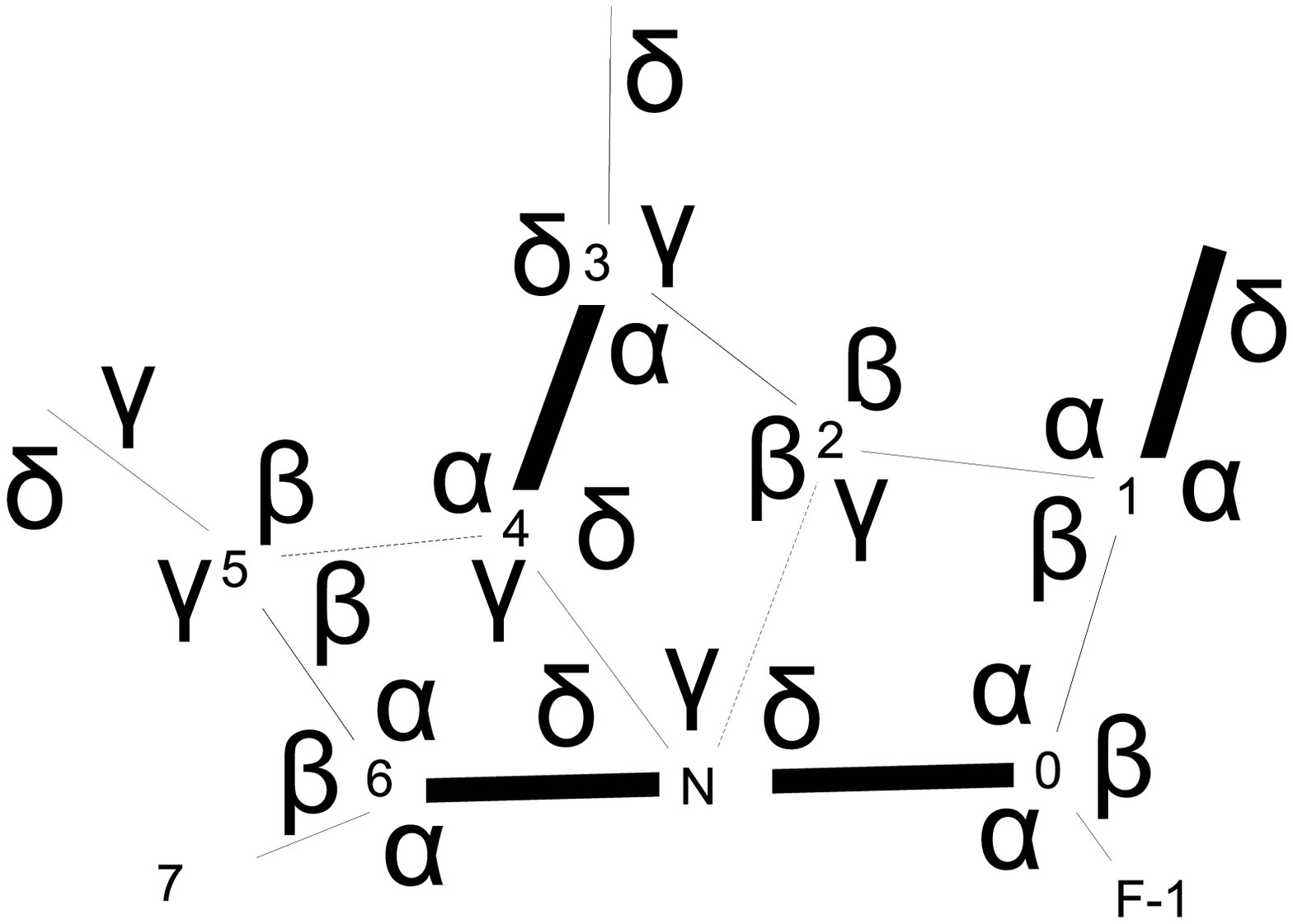}
\end{center}
\caption{The proof that Property~1 holds when the tile is of
type~2~(left), and the proof of
Claim~\ref{claim:property2}~(right). The thick edges have length
$b$.\label{fig:4}}\end{figure}

\begin{claim}\label{claim:property2} Property~2 holds when the tile is
 of type~2 or of type~4.\end{claim}

\begin{proof}Assume otherwise. See Figure~\ref{fig:4}~(right). Since the edge $N v_{6i}$ has length $b$,
 the type of the angle $\angle v_{6i+2} v_{6i+1} v_{6i}$ is $\beta$. Since the edge
 $S v_{6i+1}$ has length $b$, the type of $\angle v_{6i} v_{6i+1} S$ is
 $\alpha$. Assume the type of $\angle S v_{6i+1} v_{6i+2}$ is $\delta$.  When the
 tile is of type~4, we have a contradiction against Lemma~\ref{lem:6.3}.
 When the tile is of type~2, we have a contradiction against
 Lemma~\ref{lem:u} because $\alpha=\delta$ follows by comparing the
 vertex types of $v_{6i}$ and $v_{6i+1}$. Therefore the vertex types of $v_{6i}$
 and $v_{6i+1}$ are both
\begin{align}
2\alpha+\beta=2\pi\label{eqv01},
\end{align}
and thus the type of $\angle v_{6i+1} v_{6i+2} v_{6i+3}$ is $\beta$. If the vertex type
 of $v_{6i+2}$ is $\beta+2\gamma$, then $\alpha=\gamma$ by comparing the
 vertex types of $v_{6i+1}$ and $v_{6i+2}$. This contradicts against
 Lemma~\ref{lem:6.2}, whether the tile is of type~2 or of type~4. Thus
 the vertex type of $v_{6i+2}$ is
\begin{align}
 2\beta+\gamma=2\pi\label{eqv2}.
\end{align}
Because the type of $\angle v_{6i+2} v_{6i+3} S$ is $\gamma$, if the vertex type
 of $v_{6i+3}$ is $2\alpha+\gamma$, then \eqref{eqv01} and \eqref{eqv2} implies
 $\alpha=\beta=\gamma=2\pi/3$, which contradicts against
 Lemma~\ref{lem:6.2}. Therefore  the vertex type of $v_{6i+3}$ is
\begin{align}
 \alpha+\gamma+\delta=2\pi.\label{eqv3}
\end{align} 
Then the vertex type of $v_{6i+4}$ is \eqref{eqv3}, too.  Otherwise
$\alpha+\beta+\delta=2\pi$. This contradicts against Lemma~\ref{lem:u}
and the equation~\eqref{eqv3}.  Hence the type of $\angle N v_{6i+6}
v_{6i+5}$ is $\alpha$. The type of $\angle S v_{6i+5} v_{6i+6}$ is
$\gamma$. Otherwise it is $\beta$. Then \eqref{eqv2} implies
$\beta=\gamma$, which contradicts against Lemma~\ref{lem:u}.  Thus the
type of $\angle v_{6i+5} v_{6i+6} v_{6i+7}$ is $\beta$.

Actually, the type of $\angle v_{6i+7} v_{6i+6} N$ is
$\alpha$. Otherwise it is $\delta$, and the vertex type of $v_{6i+6}$ is
$\alpha+\beta+\delta$. This contradicts against Lemma~\ref{lem:u} and
\eqref{eqv3}.

To sum up, from the negation of Property~2 we derived the negation of
Property~2 with the indices increased by 6.  This increase of the
indices by 6 contributes $\gamma+2\delta$ to the vertex type of the
vertex $N$, and three tiles to the northern part of the
tiling. Therefore $\gamma+2\delta={2\pi}/(F/6)={12\pi}/{F}$. By solving
the system of this equation, \eqref{eqv2}, \eqref{eqv3}, and
\eqref{eq:area}, we have $0<\delta=(10-F)\pi/F$, which is, however, be
negative from $F\ge10$. This ends the proof of
Claim~\ref{claim:property2}.\end{proof}

We can prove Property~$2^*$ holds when the tile is of type~2, by the
 same argument as above with
 $(\alpha,\beta)\leftrightarrow(\delta,\gamma)$ swapped.  We can prove
 Property~$2^*$ holds when the tile is of type~4 by Lemma~\ref{lem:6.3}.
This ends the proof of Lemma~\ref{lem:lm}.
\end{proof}

Below is the continuation of the proof of
Lemma~\ref{lem:combinatorial}. The readers are kindly advised to follow
the argument with Figure~\ref{chart:a}.

\begin{proof}
By Lemma~\ref{lem:lm}, the assertion~\eqref{type2} of
 Lemma~\ref{lem:combinatorial} follows.  Hence $L(e)=L(f(e))=L(g(e))$
 for every edge $e$ of the chart.  

First consider the case the properties
 $\eqref{prop3star}^*$ and $\eqref{prop4}^*$ of Lemma~\ref{lem:6} hold.
By the property~$\eqref{prop3star}^*$, for each nonnegative integer $i<F/6$,
\begin{align}
\mbox{the type of $\angle v_{6i} N v_{6i+2}$ is $\alpha$} \label{0N2}
\end{align} 
and the type of $\angle v_{6i+2} v_{6i+1} v_{6i}$ is $\gamma$. By
Lemma~\ref{lem:6}, 
\begin{align}
\mbox{the type of $\angle v_{6i-1} v_{6i} v_{6i+1}$ is
 $\gamma$.} \label{-101}
\end{align}
By the property~$\eqref{prop4}^*$,
\begin{align}
\mbox{the type of $\angle v_{6i+2} N v_{6i+4}$ is $\beta$.} \label{2N4}
\end{align}
By \eqref{0N2},
\begin{align}
\mbox{the type of $\angle N v_{6i+2} v_{6i+1}$ is $\beta$.} \label{N21}
\end{align}
By \eqref{-101}, the type of $\angle v_{6i} v_{6i+1} S$ is
 $\delta$. By Lemma~\ref{lem:lm}, the type of 
 $\angle S v_{6i+1} v_{6i+2}$ is $\alpha$. Thus $K(\psi)=K(f(\psi))$ if $\psi$ is
 an angle around $v_{6i}$ or $v_{6i+1}$. Observe
\begin{align}
\mbox{the type of $\angle v_{6i+1} v_{6i+2} v_{6i+3}$ is $\beta$}, \label{123}
\end{align}
and thus
\begin{align}
\mbox{the type of $\angle v_{6i+2} v_{6i+3} S$ is $\gamma$}. \label{23S}
\end{align}
By \eqref{2N4}, 
\begin{align}\mbox{the type of $\angle v_{6i+3} v_{6i+2} N$ is $\gamma$.}\label{32N}
\end{align}
Thus the vertex type of $v_{6i+2}$ is \eqref{eqv2}.
Note the type of $\angle N v_{6i+4} v_{6i+3}$ is $\alpha$,
and 
\begin{align}
\mbox{the type of $\angle v_{6i+4} v_{6i+3} v_{6i+2}$ is $\delta$.}
 \label{432}
\end{align}
By the property~$\eqref{prop3star}^*$, we have the equation \eqref{eqv3}.
By this, \eqref{23S}, and \eqref{432}, if the type of
 $\angle S v_{6i+3} v_{6i+4}$ is $\delta$, then we have $\alpha=\delta$ which contradicts against
 Lemma~\ref{lem:u}. So
\begin{align}
 \mbox{the type of $\angle S v_{6i+3} v_{6i+4}$ is $\alpha$, and}\label{S34}\\
 \mbox{that of $\angle v_{6i+3} v_{6i+4} v_{6i+5}$ is $\delta$}.\label{345}
\end{align}
By the property~$\eqref{prop3star}^*$, the type of $\angle v_{6i+5} v_{6i+4} N$ is
 $\gamma$. Hence $K(\psi)=K(f(\psi))$ if $\psi$ is an angle around $v_{6i+3}$ or
 $v_{6i+4}$.  Therefore the vertex type of $v_{6i+3}$ is \eqref{eqv3}.
The property~$\eqref{prop3star}^*$ furthermore implies the type of
 $\angle v_{6i+6} v_{6i+5} v_{6i+4}$ is $\beta$. By \eqref{345}, the
 type of $\angle v_{6i+4} v_{6i+5} S$ is $\gamma$. By \eqref{-101}, the
 type of $\angle S v_{6i+5} v_{6i+6}$ is $\beta$. Therefore
 $K(\psi)=K(f(\psi))$ if $\psi$ is an angle around $v_{6i+2}$ or $v_{6i+5}$. This
 completes the proof of the assertion~\eqref{invariance} of
 Lemma~\ref{lem:combinatorial}.
By the property~$\eqref{prop3star}^*$,
\begin{align}
 \mbox{the type of $\angle N v_{6i} v_{6i-1}$ is $\alpha$ and}\nonumber\\
 \mbox{the type of $\angle v_{6i-2} N v_{6i}$ is $\delta$.} \label{4N6}
\end{align}
Because of \eqref{0N2}, \eqref{2N4} and \eqref{4N6}, the vertex type of
the north pole is 
\begin{align*}
(\alpha+\beta+\delta)\frac{F}{6}=2\pi.
\end{align*} 
From this,
equation \eqref{eq:area} and equation \eqref{eqv3}, we obtain the
assertion~\eqref{thm:k:case1} of Lemma~\ref{lem:combinatorial}.  The
 claim $\gamma>\pi$ follows from
 Lemma~\ref{lem:combinatorial}~\eqref{multiple6}.
The assertion~\eqref{fundamental} of Lemma~\ref{lem:combinatorial}
 follows from \eqref{4N6}, \eqref{0N2},  \eqref{2N4}; 
the property~$\eqref{prop3star}^*$, \eqref{-101}; 
\eqref{N21}, \eqref{123}, \eqref{32N};
\eqref{23S}, \eqref{432}, and \eqref{S34}.

Next consider the case the properties $\eqref{prop3star}$ and
$\eqref{prop4}$ of Lemma~\ref{lem:6} hold. By the same argument as
above but with $(\alpha,\beta)\leftrightarrow(\delta,\gamma)$ being
swapped, we obtain a chart from Figure~\ref{chart:a} with
$(\alpha,\beta)\leftrightarrow (\delta,\gamma)$ swapped and obtain the
assertion~\eqref{thm:k:case1}$^*$. But these data provide exactly the
same angle-assignment for the map $\pdw{F}$ we have just obtained in the
 case the properties~\eqref{prop3star}$^*$ and
 \eqref{prop4}$^*$ hold. Hence the
 assertions~\eqref{multiple6}, \eqref{type2}, \eqref{invariance},
 \eqref{thm:k:case1} and \eqref{fundamental} of
 Lemma~\ref{lem:combinatorial} hold also.
\end{proof}

\section{The absolute values of the angles, the edges and the number of tiles}\label{subsec:chart}

By trigonometry argument, we finish the proof of
Assertion~\eqref{assert:drawn} of Theorem~\ref{thm:k}. In the rest of
the paper, we use the spherical polar coordinate system introduced in
Assertion~\eqref{thm:l:coordinates} of Theorem~\ref{thm:k}.  The vertex
$S$ of the chart is located at the south pole $(\pi,0)$.  To see it,
first we can compute the longitude $\varphi$ of the vertex $S$ by
applying trigonometry arguments along a path $N v_0 v_1 S$.  By rotating
the path $N v_0 v_1 S$ by $(\alpha+\beta+\delta) = 12\pi/F$ radian, we
obtain a path $N v_6 v_7 S$. By applying the same trigonometry arguments
along the latter path, the longitude of $S$ becomes $\pi+\varphi$.  Thus
the vertex $S$ should be the north pole or the south pole of the
spherical polar coordinate system. If the vertex is the north pole, then
some tile containing the vertex $N$ and some tile containing the vertex
$S$ shares the interior, which contradicts against the definition of the
tiling.

For the length $a$ of an edge of the tile, $\cos a$ is positive.  To see it,
 assume $a\ge \pi/2$. Then the two points $v_2$ and $v_4$ are located on
 the southern hemisphere while the point $v_3$ on the northern
 hemisphere. So $\gamma=\angle v_3 v_2 N$ is less than or equal to $\pi$, which
 contradicts against the assertion~\eqref{thm:k:case1} of Lemma~\ref{lem:combinatorial}.

Consider a triangle $N v_2 v_3$. If we join the edge $N v_3$ of the
triangle with the edge $v_3 S$ of the tiling, we obtain a geodesic line
from the north pole to the south pole. Because the edge $v_3 S$ has
length $a$, 
\begin{align*}
N v_3 = \pi-a.
\end{align*}
 The first equation of \eqref{32N2N4} of
 Lemma~\ref{lem:combinatorial} implies $\angle v_3 v_2 N=\gamma>\pi$. So
 the angle between the two edges $N v_2$ and $v_2 v_3$ is $2\pi -
 \gamma=8\pi/F$.  By the spherical cosine theorem, we obtain a quadratic
 equation $\cos(\pi -a) = \cos^2 a + \sin^2 a \cos (8\pi/F)$ of $\cos
 a$. Since $\cos a>0$, the solution is
\begin{align}
\cos a = - \frac{\cos\left(\frac{8\pi}{F}\right)}{1-\cos\left(\frac{8\pi}{F}\right)} >0. \label{qq}
\end{align} 

Thus by the premise $F\ge10$ of Theorem~\ref{thm:k} and
Lemma~\ref{lem:combinatorial}~\eqref{multiple6}, the number of tiles is
$F=12$. From this and \eqref{qq}, the length $a$ is as in the
equation~\eqref{lgh:a} of Theorem~\ref{thm:k}.

By $F=12$ and the assertion~\eqref{thm:k:case1} of
Lemma~\ref{lem:combinatorial}, the inner angles $\beta$ and $\gamma$ are
as in the equations \eqref{angle:beta} and
\eqref{angle:gamma} of Theorem~\ref{thm:k}. By the spherical cosine
theorem applied to an isosceles triangle $v_2 N v_4$, the equations
\eqref{lgh:a} and \eqref{angle:beta} imply
\begin{align}
\cos v_4 v_2 = \cos ^2 a + \sin^2 a \cos
\beta
=\frac{1}{3^2} + \left( 1- \frac{1}{3^2}\right) \cos \frac{\pi}{3} = \frac{5}{9}.
\label{lgh:42}
\end{align}

The length $b$ of an edge is less than $\pi$, because otherwise the two edges
$N v_0$ and $N v_6$ of length $b$ share more than two points.  We compute the length
$b$ by considering a triangle $N v_0 v_1$. If we join the edge $N v_1$
of the triangle with the edge $v_1 S$ of the tiling, we obtain a
geodesic line from the north pole to the south pole. Moreover the edge
$v_1 S$ has length $b$. Therefore the edge $N v_1$ of the triangle $N
v_0 v_1$ has length $\pi-b$. Because two tiles $S v_{11} v_0 v_1$ and
$v_4 N v_2 v_3$ are congruent,
\begin{align*}
\mbox{the segment $v_2 v_4$ has length $\pi - b$. }
\end{align*}
By \eqref{lgh:42}, the length $b$ is as in the equation \eqref{lgh:b} of
Theorem~\ref{thm:k}.

Then we compute the other inner angles $\alpha$ and $\delta$ of the
tile.  To compute $\alpha$, we apply the spherical cosine theorem to a
triangle $N v_3 v_4$. By the equations \eqref{lgh:a}, \eqref{lgh:b} and
\eqref{32N2N4}, we have $\cos (\pi - a) = \cos b \cos a + \sin b \sin a
\cos \alpha$. Here $0<\alpha<\pi$ because $\alpha$ occurs twice in the
vertex type of the vertex $N$. Hence the inner angles $\alpha$ and
$\delta$ are as in the equations~\eqref{angle:alpha} and
\eqref{angle:delta} of Theorem~\ref{thm:k}.

By the assertion~\eqref{fundamental} and the
assertion~\eqref{invariance} of Lemma~\ref{lem:combinatorial}, all the
angles of the chart $\A$ are as in Assertion~\eqref{assert:drawn} of
Theorem~\ref{thm:k}.

\section{The existence of the tile}\label{sec:exist}

To prove that the tiling $\T$ indeed exists, we verify that the
quadrangle $N v_2 v_3 v_4$ indeed exists. It is sufficient to show the
following three assertions:
\begin{enumerate}[(i)]
\item \label{i:diag} Let $u$ be the value of $\cos v_2 v_4$ computed by
applying the spherical cosine theorem to the triangle $v_2 v_3 v_4$. Then
$u$ is equal to the value $5/9$ of $\cos v_2 v_4$ we have computed
      already in \eqref{lgh:42} by applying
spherical cosine theorem to the triangle $v_2 N v_4$.

\item
\label{ii:triangle}
The triangles $N v_2 v_4$ and $v_2 v_3 v_4$ indeed exist.

\item
\label{iii:ribbon}The
different edges of the tile $N v_0 v_1 v_2$  do not have a common inner point.
\end{enumerate}

To prove the assertion~\eqref{i:diag}, we observe that $\sin a=
2\sqrt2/3$ and $\sin b =2\sqrt{14}/9$, because $0<a,b<\pi$, the
equations \eqref{lgh:a} and \eqref{lgh:b} hold.  Then $u= \cos v_3 v_2
\cos v_3 v_4 + \sin v_3 v_2 \sin v_3 v_4 \cos \angle v_4 v_3 v_2 = \cos
a \cos b + \sin a \sin b \cos \delta = 5/9$ as desired.

To prove the assertion~\eqref{ii:triangle}, we employ
Proposition~\ref{prop:vinberg}.  First we verify the assumptions of
Proposition~\ref{prop:vinberg} for the triangle $v_2 v_3 v_4$.  Let
$\phi$ be $\angle v_0 N v_1$. By applying spherical cosine theorem to a
triangle $v_0 N v_1$, we have $\cos a = \cos b \cos(\pi -b) + \sin b
\sin(\pi -b) \cos\phi$. By the equation~\eqref{lgh:b}, we have $\phi=\pm
\arccos (13/14)$. We prove the equation~\eqref{angle:phi}, that is,
$\phi=\arccos(13/14)$. If $\phi$ is $- \arccos (13/14)$, then the
spherical polar coordinate of $v_1$ and that of $v_2$ are
$(\theta_1,\varphi_1):=(\pi -b, -\arccos(13/14))$ and
$(\theta_2,\varphi_2):=(a, \alpha)$ respectively. For the length $a$ of
the geodesic line between the two points $v_1$ and $v_2$, the value
$\cos a=1/3$ should be equal to the inner product of the cartesian
coordinates of $v_1$ and $v_2$, that is,
$\sin(\theta_1)\sin(\theta_2)\cos(\varphi_2 - \varphi_1) +
\cos(\theta_1)\cos(\theta_2) = -5/21$.  This is a contradiction.  Hence
we have \eqref{angle:phi}. Thus
\begin{align}
\angle v_2 v_4 v_3=\angle v_0 N v_1 = \phi = \arccos\frac{13}{14}=  0.380251\cdots. \label{k}
\end{align}
By spherical cosine theorem applied to the triangle $v_2 v_3 v_4$, we
have $$\cos b = \cos a \cos (\pi - b) + \sin a \sin (\pi -b) \cos \angle
v_3 v_2 v_4.$$ By the equations \eqref{lgh:a}, \eqref{lgh:b}, and
\eqref{qq}, we have $\cos\angle v_3 v_2 v_4 = -5/(2\sqrt7)$. If $\angle
v_3 v_2 v_4$ is greater than $\pi$, then the edges $v_3 v_4$ and $N v_2$
of the tile $N v_2 v_3 v_4$ has a common inner point, which is a
contradiction. Hence
\begin{align}
 \angle v_3 v_2 v_4 = \arccos \frac{-5}{2\sqrt7} =
 2.80812\cdots. \label{l}
\end{align}
For a geodesic line from $N$ to $S$ through $v_1$, we have $\angle v_0
v_1 N = \angle v_3 v_2 v_4$. So
\begin{align}
 \angle  v_4 v_3 v_2 = \delta  = \pi - \angle v_3 v_2 v_4 =  \arccos
 \frac{5}{2\sqrt7} = 0.3334373\cdots. \label{m}
\end{align}
By \eqref{k}, \eqref{l} and \eqref{m}, the assumptions of
Proposition~\ref{prop:vinberg} hold for the triangle $v_2 v_3
v_4$.

Next we verify the assumptions of Proposition~\ref{prop:vinberg} for the
triangle $v_2 N v_4$.  By the equations \eqref{angle:gamma}, \eqref{l},
\eqref{angle:beta}, \eqref{angle:alpha}, and \eqref{angle:phi}, the
inner angles of the triangle $v_2 N v_4$ are
\begin{align*}
&\gamma -\angle v_3 v_2 v_4 = \frac{4\pi}{3} - \arccos\frac{-5}{2\sqrt7}
 = 1.38067\cdots,\quad \beta = \frac{\pi}{3} = 1.0472\cdots,\\
&\alpha-\phi = \arccos\frac{-1}{2\sqrt{7}} -
 \arccos\frac{13}{14} = 1.38067\cdots.
\end{align*}
These also satisfy the assumptions of
Proposition~\ref{prop:vinberg}. Thus the triangle $v_2 N v_4$
indeed exists.

The assertion~\eqref{iii:ribbon} holds because the longitudes of the
vertices $v_0, v_1$ and $v_2$ are, respectively, $0$,
$\phi=\arccos(13/14)$, and $\alpha=\arccos(-1/(2\sqrt7))$, increasing
between $0$ and $\pi$.
To sum up, the quadrangle $N v_2 v_3 v_4$ indeed exists.

\section{The uniqueness up to special orthogonal
  transformation}\label{subsec:coordinates} 

We prove Assertion~\eqref{thm:l:coordinates} of Theorem~\ref{thm:k}.
Observe that the spherical polar coordinate of each vertex $v_{2i}$
(resp. $v_{2i+1}$) of the tiling $\T$ is determined from
Assertion~\eqref{assert:drawn} of Theorem~\ref{thm:k} by considering the
angle around the north pole $N$~(resp. around the south pole $S$ with the
equation~\eqref{angle:phi} of Theorem~\ref{thm:k}). The length of the
geodesic line between $N$ and $v_{2i+1}$ is $\pi$ subtracted by the
length of the edge $S v_{2i+1}$.

We have a freedom to set a spherical polar coordinate system, so long
that the north pole is the vertex $N$ and the ``Greenwich meridian'' of
the spherical polar coordinate system contains the edge $N v_0$.
However, the spherical coordinate system must respect the cyclic orders
of the vertices of the chart. So each vertex is unique up to special orthogonal transformation.
In fact, the mirror image of the tiling cannot be transformed to the
original tiling, by a special orthogonal transformation.

\section{The symmetry}\label{subsec:symmetry}
To complete the proof of Theorem~\ref{thm:k}, we prove
Assertion~\eqref{D2} of Theorem~\ref{thm:k}: Sch\"onflies symbol of the
tiling $\T$ is $D_2$. The Sch\"onflies symbol is determined according
to \cite[p.~55]{cotton09:_chmic_applic_of_group_theor} and Step~1,
Step~2,$\ldots$ of
\cite[Figure~3.10~(p.~56)]{cotton09:_chmic_applic_of_group_theor}.
As for Step~1 of Figure~3.10, the tiling $\T$ does not have continuous rotational
symmetry, and so the Sch\"onflies symbol is neither $C_{\infty v}$ nor
$D_{\infty h}$. Moreover the tiling has none of 3-fold and 5-fold
symmetry, so the
Sch\"onflies symbol is none of $T,T_h, T_d, O, O_h, I,
I_h$.

Observe that the tiling $\T$ has the following three $2$-fold rotation axes
perpendicular to each other~(cf.\@ Figure~\ref{fig:perpD2Axis}):
\begin{itemize}
\item The axis through the both poles. The rotation around this axis by $\pi$-radian maps a point
      $(\theta, \varphi)$ on the sphere to a point
$(\theta, \varphi+\pi)$. The action of the rotation on the vertices and the edges of the tiling $\T$ is
the automorphism $g$ defined in \eqref{auto:g}.

\item The axis through the ball's center and the midpoint of $v_0$ and
$v_1$.  The rotation around this axis by $\pi$-radian maps a point
      $(\theta, \varphi)$ on the sphere to a point
$(\pi-\theta, \phi-\varphi)$ where $\phi$ is defined in the equation \eqref{angle:phi}.  
The action of the rotation on the vertices and the edges of the tiling $\T$ is
the automorphism $f$ defined in \eqref{auto:f}.

\item  The axis through the ball's center and the midpoint of $v_3$
and $v_4$. The $2$-fold rotation around this axis corresponds to the
      automorphism $g\circ f$.
\end{itemize} 
Because of the three
perpendicular 2-fold axes, we can skip the Step~2, and we see in Step~3 that the Sch\"onflies symbol is none of $S_4, S_6,
S_8,\ldots$. Hence, we go to Step~5 and see that the symbol is one of
$D_2, D_{2d}, D_{2h}$.  But by Figure~\ref{fig:perpD2Axis}, the tiling
has no mirror plane perpendicular to a 2-fold axis, the symbol is not
$D_{2h}$. Moreover the tiling has no mirror plane between two 2-fold
axes, so the symbol is not $D_{2d}$ either. Therefore the Sch\"onflies
symbol is $D_2$.  This completes the proof of Theorem~\ref{thm:k}.

Three movies of the tiling $\T$ spinning around the rotation axes are
available at a web-site http://www.math.tohoku.ac.jp/akama/stcq.

\section*{Acknowledgement}
The author thanks Kosuke Nakamura and an anonymous referee.


\begin{thebibliography}{10}

\bibitem{akama11:_spher_tilin_by_congr_quadr_pdw_ii}
Y.~Akama and K.~Nakamura.
\newblock Spherical tilings by congruent quadrangles over pseudo-double
  wheels~({II}) | the ambiguity of the inner angles, 2012.
\newblock Preprint.

\bibitem{akama11:_spher_tilin_by_congr_quadr_pdw_iii}
Y.~Akama and Y.~Sakano.
\newblock Classification of spherical tilings by congruent quadrangles over
  pseudo-double wheels~({III}) | the essential uniqueness in case of convex
  tiles, 2011.
\newblock Preprint.

\bibitem{akama12:_spher_tilin_by_congr_rhombi_kites_darts}
Y.~Akama and Y.~Sakano.
\newblock Classification of spherical tilings by congruent rhombi~(kites,
  darts).
\newblock In preparation.

\bibitem{MR1254932}
D.~V. Alekseevskij, {\`E}.~B. Vinberg, and A.~S. Solodovnikov.
\newblock Geometry of spaces of constant curvature.
\newblock In {\em Geometry, {II}}, Vol.~29 of {\em Encyclopaedia Math. Sci.},
  pp.~1--138. Springer, Berlin, 1993.

\bibitem{MR2186681}
G.~Brinkmann, S.~Greenberg, C.~Greenhill, B.~D. McKay, R.~Thomas, and P.~Wollan.
\newblock Generation of simple quadrangulations of the sphere.
\newblock {\em Discrete Math.}, Vol.~305, No.~1-3, pp.~33--54, 2005.

\bibitem{MR2410150}
J.~H. Conway, H.~Burgiel, and C.~Goodman-Strauss.
\newblock {\em The Symmetries of Things}.
\newblock A K Peters Ltd., Wellesley, MA, 2008.

\bibitem{cotton09:_chmic_applic_of_group_theor}
F.~A.~Cotton.
\newblock {\em Chemical Applications of Group Theory}.
\newblock Wiley India, third edition, 2009.

\bibitem{MR2744811}
R.~Diestel.
\newblock {\em Graph theory}, Vol.~173 of {\em Graduate Texts in
  Mathematics}.
\newblock Springer, Heidelberg, fourth edition, 2010.

\bibitem{sakano10:_towar_class_of_spher_tilin}
Y.~Sakano.
\newblock Toward classification of spherical tilings by congruent quadrangles.
\newblock Master's thesis, Mathematical Institute, Tohoku University, March,
  2010.

\bibitem{agaoka:quad}
Y.~Ueno and Y.~Agaoka.
\newblock Examples of spherical tilings by congruent quadrangles.
\newblock {\em Mem. Fac. Integrated Arts and Sci., Hiroshima Univ. Ser. {IV}}, Vol.~27, pp.~135--144,
  2001.

\bibitem{MR1954054}
Y.~Ueno and Y.~Agaoka.
\newblock Classification of tilings of the 2-dimensional sphere by congruent
  triangles.
\newblock {\em Hiroshima Math. J.}, Vol.~32, No.~3, pp.~463--540, 2002.
\end{thebibliography}
\end{document}